\documentclass[11pt,a4paper]{article}
\usepackage{a4}

\usepackage{mathrsfs} 
\usepackage{amssymb} 
\usepackage{amsmath} 
\usepackage{amsthm} 
\usepackage{graphicx}
\usepackage{hyperref}

\newtheorem{theorem}{Theorem}
\newtheorem{proposition}[theorem]{Proposition}
\newtheorem{lemma}[theorem]{Lemma}
\newtheorem{conjecture}[theorem]{Conjecture}
\newtheorem{claim}[theorem]{Claim}

\newcommand{\eps}{\varepsilon}

\renewcommand{\theenumi}{{\rm(\roman{enumi})}}

\title{Colouring squares of claw-free graphs}

\author{
R\'emi de Joannis de Verclos
\and
 Ross J. Kang
 \and
 Lucas Pastor
}

\begin{document}
\maketitle

\begin{abstract}
Is there some absolute $\varepsilon > 0$ such that for any claw-free graph $G$, the chromatic number of the square of $G$ satisfies $\chi(G^2) \le (2-\varepsilon) \omega(G)^2$, where $\omega(G)$ is the clique number of $G$? Erd\H{o}s and Ne\v{s}et\v{r}il asked this question for the specific case of $G$ the line graph of a simple graph and this was answered in the affirmative by Molloy and Reed. We show that the answer to the more general question is also yes, and moreover that it essentially reduces to the original question of Erd\H{o}s and Ne\v{s}et\v{r}il.
\end{abstract}

\let\thefootnote\relax

\let\thefootnote\relax


\section{Introduction}\label{sec:intro}

Let $G$ be a claw-free graph, that is, a graph that does not contain the complete bipartite graph $K_{1,3}$ as an induced subgraph. We consider the square $G^2$ of $G$, which is the graph formed from $G$ by the addition of edges between those pairs of vertices connected by some two-edge path in $G$, and consider proper colourings of $G^2$.
In particular, we relate the chromatic number $\chi(G^2)$ of $G^2$ to the clique number $\omega(G)$ of $G$.
Our main result is the following.

\begin{theorem}\label{thm:main}
There is an absolute constant $\eps > 0$ such that $\chi(G^2) \le (2-\eps)\omega(G)^2$ for any claw-free graph $G$. 
\end{theorem}

\noindent
This extends a classic result of Molloy and Reed~\cite{MoRe97}. Their work is an acclaimed combination of structural and probabilistic methodology that established the special case for Theorem~\ref{thm:main} of $G$ the line graph $L(F)$ of some (simple) graph $F$. Note that $\omega(G) = \Delta(F)$ here (unless $F$ is a disjoint union of paths, cycles and at least one triangle), where $\Delta(F)$ denotes the maximum degree of $F$.

Claw-free graphs constitute an important superclass of the class of line graphs. As such, there have been sustained efforts in combinatorial optimisation to extend results from the smaller to the larger class, especially for stable sets (which are matchings in an underlying graph of the given line graph), the starting point being the seminal work of Edmonds~\cite{Edm65}, cf.~e.g.~\cite{Min80,Sbi80,NaTa01,EOSV08,FOGS14}.
Salient to our work, we point out that significant efforts have also been made for proper colourings (which are proper edge-colourings in an underlying graph of the given line graph), the starting point being the classic Gupta--Vizing theorem~\cite{Gup66,Viz64}, cf.~e.g.~\cite{Sum81,Kie89,ChOv07,ChSe10,KiRe13,KiRe15}.

Along similar lines, our starting point is a notorious problem from the 1980s on strong edge-colourings due to Erd\H{o}s and Ne\v{s}et\v{r}il, cf.~\cite{Erd88}. Having in mind the line graph of a five-cycle each of whose vertices has been substituted with a stable set, they conjectured the following assertion restricted to $G$ a line graph.

\begin{conjecture}\label{conj:main}
For any claw-free graph $G$, 
\begin{align*}
\chi(G^2) \le 
\begin{cases} \frac54\omega(G)^2 &\text{if $\omega(G)$ even}, \\ 
\frac14(5\omega(G)^2 -2\omega(G)+1) & \mbox{otherwise.} \end{cases} 
\end{align*}
\end{conjecture}

\noindent
The conjecture of Erd\H{o}s and Ne\v{s}et\v{r}il remains open in general. Theorem~\ref{thm:main} provides evidence towards the stronger conjecture we have cheekily just introduced.

Several sharp claw-free graph results for proper colourings have been established; however, while important results on stable sets have extended quite well from line graphs to claw-free graphs (albeit thanks to serious, continuing work spanning decades), strictly the same cannot be said for proper colourings. To illustrate, the classic result on edge-colouring due to Vizing~\cite{Viz64} and, independently, Gupta~\cite{Gup66}, implies that $\chi(G) \in \{\omega(G), \omega(G)+1\}$ for $G$ a line graph.
On the other hand, by considering large triangle-free graphs without large stable sets~\cite{Kim95,AKS80}, one sees (cf.~\cite{ChSe10}) that $\sup\{\chi(G) \ | \ \text{$G$ claw-free, $\omega(G) = \omega$}\} = \Omega(\omega^2/(\log \omega)^2)$ as $\omega \to \infty$. So any overall upper bound on $\chi(G)$ in terms of $\omega(G)$ must be worse for $G$ a claw-free graph than for $G$ a line graph.
One might argue that this is intuitive from the fact that proper colourings (proper edge-colourings) are more complicated combinatorial structures than stable sets (matchings).

To our surprise, for proper colourings of the square (which are yet more complicated combinatorial structures), proceeding from line graphs to claw-free graphs, the situation does not worsen in the same sense. Rather, within the class of claw-free graphs the worst cases are likely to be line graphs. This motivates the direct extrapolation of the conjecture of Erd\H{o}s and Ne\v{s}et\v{r}il to Conjecture~\ref{conj:main}.
Let us make these vague sentiments more precise.

We prove the following three results.
Recall that a graph is a quasi-line graph if every neighbourhood induces a subgraph that can be covered by two cliques.

\begin{theorem}\label{thm:clawfree}
  For any claw-free graph $G$, either $G$ is a quasi-line graph or there is a vertex $v$ with square degree $\deg_{G^2}(v) \le \omega(G)^2 + (\omega(G)+1)/2$
  whose neighbourhood $N_G(v)$ induces a clique in $(G \setminus v)^2$.
\end{theorem}
\begin{theorem}\label{thm:quasiline}
  For any quasi-line graph $G$, either $G$ is the line graph of a multigraph
  or there is a vertex $v$ and a set $S \subseteq N_G(v)$
  such that every vertex $u \in S \cup \{ v \}$ has square degree
  $\deg_{G^2}(u) \le \omega(G)^2+\omega(G)$ and
  $N_G(v) \setminus S$ induces a clique in $(G \setminus v)^2$.
\end{theorem}
\begin{theorem}\label{thm:linemulti}
Theorem~\ref{thm:main} holds for $G$ being the line graph of a multigraph.
\end{theorem}

By a simple greedy procedure that always colours the vertex with least square degree, Theorem~\ref{thm:main} follows from the above three results. We spell out this procedure in Section~\ref{sec:main}.
Moreover, by the same greedy approach Theorems~\ref{thm:clawfree} and~\ref{thm:quasiline} together imply that Conjecture~\ref{conj:main} is established if it can be shown for all claw-free graphs $G$ with $\omega(G)<6$ as well as for all multigraph line graphs $G$ with $\omega(G)\ge 6$.
The proofs of Theorems~\ref{thm:clawfree} and~\ref{thm:quasiline} rely on a good structural understanding of claw-free and quasi-line graphs~\cite{ChSe05,ChSe10,ChSe12}, while the proof of Theorem~\ref{thm:linemulti} relies on a probabilistic colouring result, following~\cite{MoRe97}.

Clearly our belief is that expanding the scope beyond line graphs will not lead to claw-free examples having larger square chromatic number compared to clique number. 
This is bolstered by
the following result for which the line graphs of (simple) blown-up five-cycles are the unique extremal examples.

\begin{theorem}\label{thm:diametertwo}
For any claw-free graph $G=(V,E)$ with $\omega(G) \ge 6$, if $\omega(G^2)=|V|$, then 
\begin{align*}
|V| \le
\begin{cases} \frac54\omega(G)^2 &\text{if $\omega(G)$ even}, \\ 
\frac14(5\omega(G)^2 -2\omega(G)+1) & \mbox{otherwise.} \end{cases} 
\end{align*}
\end{theorem}
\noindent
This extends a result of Chung, Gy\'arf\'as, Tuza and Trotter~\cite{CGTT90}.

\medskip
{\em Note added.} In follow-up work, Cames van Batenburg and the second author~\cite{CaKa16+}, using techniques somewhat different from those used here, showed Conjecture~\ref{conj:main} to hold in the case $\omega(G)=3$ and moreover showed a result that together with a result here implies that Theorem~\ref{thm:diametertwo} holds in cases $\omega(G)\in\{3,4\}$.

\subsection{Plan of the paper}

Our paper is organised as follows. In the next subsection, we set some of the notation and two simple results we use. In Section~\ref{sec:clawfree}, we prove Theorem~\ref{thm:clawfree}. We use structural results to prove Theorem~\ref{thm:quasiline} in Section~\ref{sec:quasiline}. In Section~\ref{sec:linemulti}, we apply a sparsity colouring lemma to prove Theorem~\ref{thm:linemulti}. In Section~\ref{sec:main}, we prove Theorem~\ref{thm:main}. We prove the extremal result Theorem~\ref{thm:diametertwo} in Section~\ref{sec:diametertwo}.

\subsection{Notation and preliminaries}

Let $G = (V,E)$ be a (multi)graph. 
For any $v\in V$, we denote the neighbourhood of $v$ by $N_G(v)$ ($= \{w\in V \,|\, vw\in E\}$) and the degree of $v$ by $\deg_G(v)$ ($= |N_G(v)|$).
For any $U\subseteq V$, we denote the neighbourhood of $U$ by $N_G(U)$ ($= \cup_{v\in U} N_G(v) \setminus U$).
The {\em second neighbourhood} $N^2_G(v)$ of $v$ is the set of vertices at distance exactly two from $v$, i.e.~$N^2_G(v)= N_{G^2}(v) \setminus N_{G}(v)$. 
For $A,B\subseteq V$, let $E_G(A,B)$ denote the edges in the bipartite sub(multi)graph induced between $A$ and $B$, i.e.~$E_G(A,B) = \{vw\in E \,|\, v\in A, w\in B\}$ and let $E_G(A)$ denote the edges in the sub(multi)graph of $G$ induced by $A$.
Where there is no possibility of confusion, we usually drop the subscripts.
Note that the square degree $\deg_{G^2}(v)$ of $v$ equals $\deg_G(v) + |N^2_G(v)|$.
For $A,B\subseteq V$, we say $A$ is {\em complete} (resp.~{\em anti-complete}) {\em to} $B$ if all possible (resp.~no) edges between $A$ and $B$ are present.

We shall use the following observation often.

\begin{lemma}\label{lem:cliquesecond}
Let $G = (V,E)$ be a claw-free graph. Let $v\in V$ and $u\in N(v)$.
Then $N(u) \cap N^2(v)$ is a clique and $|N(u) \cap N^2(v)| \le \omega(G) - 1$.
\end{lemma}
\begin{proof}
  If $x,y \in N(u) \cap N^2(v)$ are not adjacent, then the subset $\{u, v, x, y\}$ is a claw, a contradiction. It follows that $(N(u) \cap N^2(v))\cup \{u\}$ is a clique and has at most $\omega(G)$ vertices.
\end{proof}

Let us first show how this observation together with Theorem~\ref{thm:clawfree} yields a slightly weaker version of Theorem~\ref{thm:main} (weaker when $\omega(G)$ is large) by way of a simple greedy procedure. This agrees with the ``trivial'' bound $\chi(L(F)^2) \le 2\Delta(F)^2-2\Delta(F)+1$ for any (multi)graph $F$.

\begin{proposition}\label{prop:trivial}
  Let $G$ be a claw-free graph. 
  Then $\chi(G^2) \le 2\omega(G)^2-2\omega(G)+1$.
\end{proposition}

\begin{proof}
The statement is trivial if $\omega(G) \le 2$, so assume $\omega(G)\ge 3$.
We proceed by induction on the number of vertices in $G$. The base case of $G$ having $3$ vertices is trivially true. So now assume $G$ has at least $4$ vertices and that the result holds for all graphs with fewer vertices than $G$ has.
Note that $\omega(G)^2+(\omega(G)+1)/2+1 \le 2\omega(G)^2-2\omega(G)+1$
if $\omega(G)\ge 3$.

If $G$ is not a quasi-line graph, then let $v$ be the vertex given by Theorem~\ref{thm:clawfree}. Since $G\setminus v$ is a claw-free graph and $\omega(G\setminus v) \le \omega(G)$, it follows by induction that there is a proper colouring of $(G\setminus v)^2$ with $2\omega(G)^2-2\omega(G)+1$ colours. In this colouring, necessarily all the vertices in $N_G(v)$ have different colours. Since $\deg_{G^2}(v) \le \omega(G)^2+(\omega(G)+1)/2$, there is at least one colour available to $v$ that is different from all the colours appearing on $N_{G^2}(v)$. Giving this colour to $v$ yields the desired proper colouring of $G^2$.
\end{proof}


\section{Claw-free graphs to quasi-line graphs}\label{sec:clawfree}

In this section, we prove Theorem~\ref{thm:clawfree}.

\begin{proof}[Proof of Theorem~\ref{thm:clawfree}]
  Let us write $G = (V,E)$ and $\omega(G) =\omega$ and let $v\in V$.
  In this proof, by $N(v)$ and $N^2(v)$ we mean $N_G(v)$ and $N_G^2(v)$ respectively.
  We may assume $\omega\ge 3$, otherwise the statement is trivially true.
  We may also assume without loss of generality that $G$ is connected.
  
  Our first task is to show that either $\deg_{G^2}(v) \le \omega^2+(\omega+1)/2$ or
  the subgraph induced by $N(v)$ can be covered by two cliques of $G$.

  If $G$ has no stable set of size $3$,
  then $|V|$ is less than
  the off-diagonal Ramsey number $R(3,\omega+1)$, which satisfies $R(3,\omega+1)=9$ if $\omega=3$ and $R(3,\omega+1) \le \binom{\omega+2}{2}$ otherwise~\cite{ErSz35}.
  Since $\deg_{G^2}(v) \le |V| - 1$, we have $\deg_{G^2}(v) \le 7$ if $\omega=3$ and $\deg_{G^2}(v)\le \binom{\omega+2}{2}-2$ otherwise, which in either case is at most $\omega^2+(\omega+1)/2$.
  
  Now we may assume that $G$ has a stable set of size $3$, in which case it follows from Result~8.2 of~\cite{ChSe05}
  that $\deg(v) \le 4(\omega - 1)$.
  Further, we may also assume that $N^2(v)$ is non-empty
  because otherwise $\deg_{G^2}(v) = \deg(v) \le 4(\omega - 1)$
  which is at most $\omega^2+(\omega+1)/2$
  since $\omega \ge 3$.

  Let $k$ be the largest integer such that every vertex of $N^2(v)$ has at least
  $k$ (parent) neighbours in $N(v)$ and let $u \in N^2(v)$ be a vertex attaining this
  minimum, i.e. such that $|N(u) \cap N(v)| = k$.
  Let $w \in N(u) \cap N(v)$ and
  consider the partition $N(v) = X \cup C_1 \cup C_2$ defined by
  $X = N(u) \cap N(v)\setminus \{w\}$ and
  $C_1 = (N(v) \cap N(w) \setminus X) \cup \{ w \}$.
  (So $C_2 = N(v) \setminus (N(u) \cup N(w))$.)
  We claim that both $C_1$ and $C_2$ are cliques of $G$.
  Indeed, if $x$ and $y$ are non-adjacent vertices of $C_1$
  then $w \notin \{x,y\}$ and further
  $\{x, y, w, u\}$ induces a claw, while if $x$ and $y$ are non-adjacent vertices of $C_2$ 
  then $\{x, y, w, v\}$ induces a claw, contradicting the assumption on $G$.

  We now estimate the number of paths of length two in $G$ from $v$ to $N^2(v)$.
  By Lemma~\ref{lem:cliquesecond}, there are at most $(\omega - 1) \deg(v)$
  such paths. By our assumption every vertex of $N^2(v)$ is the endpoint of
  at least $k$ of these paths and so $|N^2(v)|$ is at most $\frac{1}{k}(\omega - 1)\deg(v).$
  Consequently,
  \begin{align}\label{eq:twopaths}\tag{$\ast$}
    \deg_{G^2}(v) = \deg(v) + |N^2(v)|
    \le \left(1 + \frac{\omega - 1}{k}\right)\deg(v).\end{align}
  
  We distinguish three cases depending on the value of $k$.

\paragraph{\textbf{$k=1$:}}
  The set $X$ is empty and it follows that
  $N(v)$ induces a subgraph that can be covered by two cliques, namely $C_1$ and $C_2$.
\paragraph{\textbf{$2 \le k \le 2(\omega - 1)$:}}
  For $i\in\{1,2\}$, $|C_i| \le \omega - 1$ since $C_i \cup \{ v \}$
  is a clique. So 
  \begin{align*}\deg(v) = |C_1| + |C_2| + |X| \le 2(\omega - 1) + k - 1.\end{align*}
  Hence~\eqref{eq:twopaths} gives
  \begin{align*}
    \deg_{G^2}(v) 
    & \leq \left(1 + \frac{\omega - 1}{k}\right)(2(\omega - 1) + k - 1) =: f(k).
  \end{align*}
  The above expression $f(k)$ is a convex function of $k$ for $2 \le k \le 2(\omega - 1)$ so $\deg_{G^2}(v)$ is at most $\max\{f(2),f(2(\omega - 1))\}$. It remains to check that
  \begin{align*} f(2) = \left(1 + \frac{\omega-1}{2}\right)(2\omega - 1)
    = \omega^2+\frac{\omega-1}{2}
    < \omega^2+\frac{\omega+1}{2} \\
    \text{and } \quad
  f(2(\omega-1)) =
    6 \omega - \frac{15}{2} < \omega^2+\frac{\omega+1}{2}\end{align*}
  which is true since $\omega \ge 3$.

  \paragraph{\textbf{$k \ge 2(\omega - 1) + 1$:}}
  Together with the fact that $\deg(v) \leq 4(\omega - 1)$,~\eqref{eq:twopaths} yields
  \begin{align*}
    \deg_{G^2}(v) \le \left(1 + \frac{\omega - 1}{2(\omega - 1) + 1}\right)\cdot4(\omega - 1)
    < 6(\omega-1).
  \end{align*}
  Consequently,
  $\deg_{G^2}(v) \le 6\omega - 7 \le \omega^2+(\omega+1)/2$
  since $\omega \ge 3$.

    Our second task is to prove that, if the neighbourhood $N(v)$ does not
    form a clique in $(G \setminus v)^2$, then it is covered by two cliques (of $G$).
    Assume to the contrary that $v$ has two neighbours $x$ and $y$
    at distance at least $3$ in $G \setminus v$.
    Every other neighbour $z$ of $v$ has to be adjacent to either $x$ or $y$
    (otherwise $\{v,x,y,z\}$ induces a claw), but cannot be adjacent to
    both of them (otherwise the distance between $x$ and $y$ is at most $2$).
    It follows that $N(v)$ is covered by the union of
    $N(v) \cap N(x)$ and $N(v) \cap N(y)$.
    It remains to see that each of these sets is a clique because any
    non-edge $uw$ in $N(v) \cap N(x)$ (resp.~$N(v) \cap N(y)$) would
    give a claw $\{v, y, u, w\}$ (resp.~$\{v, x, u, w\}$).
\end{proof}


\section{Quasi-line graphs to line graphs of multigraphs}\label{sec:quasiline}

In this section, we prove Theorem~\ref{thm:quasiline}.

We rely on a known structural description of quasi-line graphs, for which we next give the necessary definitions. For further details and discussion, see~\cite{ChSe05}.

Let $G=(V,E)$ be a graph. A \emph{homogeneous set} is a set $S \subseteq V$ such that each vertex in $V \setminus S$ is adjacent either to all vertices in $S$, or to no vertex in $S$.
A \emph{homogeneous pair of cliques} is a pair $(A, B)$ of disjoint cliques such that either
$|A| \ge 2$ or $|B| \ge 2$, $A$ is a homogeneous set in $G[V \setminus B]$
and $B$ is a homogeneous set in $G[V \setminus A]$.

A {\em circular interval graph}
is any graph obtained from the following construction. 
Let $\Sigma$ be a circle and
$F_1, \ldots, F_k \subseteq \Sigma$ be a set of intervals each homeomorphic to the
interval $[0, 1]$. Let the vertex set be a finite set of points of $\Sigma$ and add an edge
between any two points if and only if they are both contained in $F_i$
for some $i \in \{1, \ldots, k\}$.
A \emph{linear interval graph} is defined in the same way, except that $\Sigma$ is a line
instead of a circle.
Observe that circular and linear interval graphs are quasi-line graphs.

A \emph{strip} $(G, a, b)$ consists of a claw-free graph $G$ and two
vertices $a$ and $b$ of $G$ such that $N_G(a)$ and $N_G(b)$ are cliques. The specified vertices $a$ and $b$ are called the {\em ends} of the strip.
In the particular case where $G$ is a linear interval graph and admits a representation in a line $\Sigma$ such that the vertices of $G$ in order along $\Sigma$ are $v_1,\dots,v_n$, we call $(G, v_1, v_n)$ a \emph{linear interval strip}.

The following operation combines two strips $(G_1, a_1, b_1)$ and $(G_2, a_2, b_2)$ to produce a claw-free graph.
Let $A_1 = N_{G_1 \setminus b_1}(a_1)$, $B_1 = N_{G_1 \setminus a_1}(b_1)$, $A_2 = N_{G_2 \setminus b_2}(a_2)$ and $B_2 = N_{G_2 \setminus a_2}(b_2)$.
 The graph obtained from the disjoint union of $G_1
\setminus \{a_1, b_1\}$ and $G_2 \setminus \{a_2, b_2\}$ by adding all possible
edges between $A_1$ and $A_2$ and all possible edges between $B_1$ and $B_2$ is called the {\em composition} of $(G_1, a_1, b_1)$ and $(G_2, a_2, b_2)$.
This graph is claw-free.

We combine $k\ge 3$ strips in the following way.
Let $G_0$ be a disjoint union of complete graphs on vertex set $\{a_1, \ldots, a_k, b_1, \ldots, b_k\}$.
For each $i \in \{1, \ldots, k\}$ let $(G_{i}', a_{i}', b_{i}')$ be a strip and
let $G_i$ be the graph obtained by composing $(G_{i-1}, a_i, b_i)$ and
$(G_{i}', a_{i}', b_{i}')$. The ultimate (claw-free) graph $G_k$ is called a \emph{composition}
of the strips $(G_1', a_1', b_1'),\dots,(G_k', a_k', b_k')$.

We apply the following structural result for the class of quasi-line graphs.
\begin{theorem}[Chudnovsky and Seymour~\cite{ChSe12}]\label{thm:decomposition}
Suppose $G$ is a connected quasi-line graph. Then one of the following must hold:
\begin{enumerate}
\item\label{first}
$G$ has a homogeneous pair of cliques,
\item\label{second}
$G$ is a circular interval graph, or
\item\label{third}
$G$ is a composition of linear interval strips.
\end{enumerate}
\end{theorem}

In fact, we need a small refinement of Theorem~\ref{thm:decomposition}.

\begin{proposition}\label{remark}
  Theorem~\ref{thm:decomposition} remains true if~\ref{first}
  is instead the following:
  \begin{enumerate}
  \item
    $G$ has a homogeneous pair $(A,B)$ of cliques where $A$ is not a homogeneous set.
  \end{enumerate}
\end{proposition}
\begin{proof}
First note that if $C$ and $C'$ are maximal among cliques that are homogeneous sets, then either they are disjoint or they are equal.
  Thus we may partition the vertex set $V$ of $G$ as $V=C_1 \cup\dots\cup C_m$
  where $C_1,\dots,C_m$ are maximal among cliques that are homogeneous sets of $G$.
  (It is allowed for $|C_i|=1$.)
  Let $\tilde{G} = (\tilde{V},\tilde{E})$ be the quotient graph of $G$ with respect to this partition,
  i.e.~the graph on $\tilde{V}=\{1,\dots,m\}$ such that $ij \in \tilde{E}$
  (resp.~$ij \notin \tilde{E}$)
  if and only if $C_i$ is complete (resp.~anti-complete) to $C_j$ in $G$.
  It must be that $\tilde{G}$ is a connected quasi-line graph, or else $G$ is not a connected quasi-line graph.
  Moreover, $\tilde{G}$ has no clique of size two that is a homogeneous set.
  
  Given $I \subseteq \tilde{V}$ and a graph $H$ on $I$, we define $\mathcal{B}(I) = \bigcup_{i \in I}C_i$ as well as a graph $\mathcal{B}(H)$ on $\mathcal{B}(I)$ as follows: let $C_i$ induce a clique in $\mathcal{B}(H)$ for every $i\in I$, and include $uv$ as an edge of $\mathcal{B}(H)$ for every pair $(u,v)\in C_i\times C_j$ if $ij$ is an edge of $H$.
  Note that $\mathcal{B}(\tilde{G})=G$.
  
  Theorem~\ref{thm:decomposition} applied to $\tilde{G}$ yields three possibilities:
  \begin{enumerate}
  \item[\ref{first}] $\tilde{G}$ has a pair $(A, B)$ of homogeneous cliques. Then $(\mathcal{B}(A),\mathcal{B}(B))$ is a homogeneous pair of cliques of $G$. We may assume that $|A| \ge 2$ and, since $\tilde{G}$ has no clique of size two that is a homogeneous set, $A$ by itself is not a homogeneous set of $\tilde{G}$. Thus $\mathcal{B}(A)$ is not a homogeneous set of $G$.
  \item[\ref{second}] $\tilde{G}$ is a circular interval graph.
    Then $G$ is too. 
  \item[\ref{third}] $\tilde{G}$ is the composition of $k$ linear interval strips
    $(G_1', a_1', b_1'),\dots,(G_k', a_k', b_k')$ over the vertex set $\{a_1,\dots,a_k,b_1,\dots,b_k\}$.
    Then $G$ is the composition of the linear interval strips
    $(G_1'', a_1', b_1'),\dots,(G_k'', a_k', b_k')$ following the same scheme
     where $G_i''$ is defined as the graph $\mathcal{B}(G_i' \setminus \{a_i',b_i'\})$
    to which we add the vertices $a_i'$ and $b_i'$ with neighbourhoods
    $\mathcal{B}(N_{G_i'}(a_i'))$ and $\mathcal{B}(N_{G_i'}(b_i'))$ respectively.
    \qedhere
  \end{enumerate}
\end{proof}

To prove Theorem~\ref{thm:quasiline},
we require the following bound on the maximum square degree of circular
interval graphs.

\begin{lemma}\label{lem:circular}
For any vertex $v$ of a circular interval graph $G$, $\deg_{G^2}(v) \le 4\omega(G)-4$.
\end{lemma}
\begin{proof}
Write $G=(V,E)$ and let $v\in V$.
  Let $\Sigma$ be a circle and let $F_1,\dots,F_k\subseteq\Sigma$ be intervals
  homeomorphic to $[0,1]$ that represent $G$.
In other words, $V$ is a subset of $\Sigma$ such that
$uw\in E$ if and only if $u$ and $w$ are both contained in $F_i$
  for some $i \in \{1,\dots, k\}$.

Note then that for every $w \in N_{G^2}(v)$ there exist $i_w,j_w \in \{1, \dots, k\}$ such that a closed interval $I_w\subseteq\Sigma$ with endpoints $v$ and $w$ is contained in $F_{i_w} \cup F_{j_w}$.
  Take $w_c$ and $w_{cc}$ to be those among the elements of $N_{G^2}(v)$ having largest intervals $I_{w_c}$ and $I_{w_{cc}}$ in, respectively, clockwise and counterclockwise direction from the perspective of $v$. Then $N_{G^2}(v)$ is contained in $F_{i_{w_c}} \cup F_{j_{w_c}} \cup F_{i_{w_{cc}}} \cup F_{j_{w_{cc}}}$ which implies that it can be covered by four cliques of $G$ such that $v$ and two distinct vertices of $N_{G^2}(v)$ each belong to more than one of the cliques. This implies the required bound.
\end{proof}

The following bound is obtained in a similar way.

\begin{lemma}\label{lem:interval}
  Let $(G, a, b)$ be a linear interval strip.
  For any vertex $v \in N_{G}(a)$,
  $\deg_{G^2}(v) \le 3\omega(G)-3$.
\end{lemma}
\begin{proof}
Write $G=(V,E)$.
  Let $F_1,\dots,F_k\subseteq\Sigma$ be closed intervals of $\mathbb{R}$
  that represent $G$. So $V$ is a subset of $\mathbb{R}$
  with minimum $a$ and maximum $b$ such that $uw\in E$
  if and only if $u$ and $w$ are both contained in $F_i$
  for some $i \in \{1,\dots, k\}$.

  Let $u$ be the vertex of $N_{G^2}(v)$ with the largest value
  and $w \in V$ such that $vw$ and $wu$ are edges.
  Let $F_{i_1}$ be an interval that contains both $a$ and $v$.
  Similarly, let $F_{i_2}$ and $F_{i_3}$ be the intervals containing
  both $v$ and $w$ and both $w$ and $u$ respectively.
  Then $N_{G^2}(v)$ is contained in the union of the cliques
  $F_{i_1} \cap V$, $F_{i_2} \cap V$ and $F_{i_3} \cap V$ since
  $F_{i_1} \cup F_{i_2} \cup F_{i_3}$ covers the interval $[a,u]$.
  Moreover, $v$ and $w$ are both elements of at least two of these cliques.
  This implies the required bound.
\end{proof}

We can now proceed to the proof of Theorem~\ref{thm:quasiline}.

\begin{proof}[Proof of Theorem~\ref{thm:quasiline}]
Let us write $G=(V,E)$ and $\omega(G)=\omega$.
  We may assume $\omega\ge 3$, otherwise the statement is trivially true.
  Since we can consider components independently, we may assume that $G$ is connected.
  
  Let us call a vertex $v$ \emph{degenerate} if $\deg_{G^2}(v) \leq \omega^2+\omega$.
  
  By Theorem~\ref{thm:decomposition} and Proposition~\ref{remark}, there are three cases to consider.

\begin{itemize}
\item[\ref{first}] $G$ contains a homogeneous pair $(A,B)$ of cliques 
  and there exist $a_1,a_2\in A$ and $b \in B$ such that $a_1b\in E$ and $a_2b\notin E$.
\end{itemize}
We first prove that every vertex of $A \cup B$ is degenerate.
Let $A'$ and $B'$ denote the set of vertices in $V\setminus (A \cup B)$ that are connected to
(all) vertices of $A$ and $B$ respectively and
set $C = V \setminus (A \cup B \cup A' \cup B')$.
Notice there is no edge from $A' \cap B'$ to $C$ as
such an edge together with $a_2$ and $b$ would form a claw.
It follows that every vertex $c \in C$ in the second neighbourhood
of any $b_0 \in B$
has a neighbour in $B' \setminus A'$,
i.e $N_{G^2}(b_0) \cap C \subseteq N(B'\setminus A') \cap C$.
So
\begin{align*}
  \deg_{G^2}(b_0) \le |N(B'\setminus A') \cap C| + |B| + |B'| + |A\cup(A'\setminus B')|.
\end{align*}
  It remains to bound the terms of this sum.
  We claim that $B\cup(B'\setminus A')$ is a clique.
  Indeed, if there is a non-adjacent pair $x,y \in B' \setminus A'$
  then $\{x, y, a_1, b\}$ induces a claw.
  This shows that $|B| \le \omega - |B' \setminus A'|$
  and $|B' \setminus A'| \le \omega - 1$.
  
  Similarly, $A\cup(A'\setminus B')$ is a clique because any non-adjacent
  pair $x,y \in A'$ forms a claw $\{x, y, a_1, b\}$.
  This proves that $|A\cup(A'\setminus B')| \le \omega$.
  
  By Lemma~\ref{lem:cliquesecond}, every $b' \in B'$
  has at most $\omega - 1$ neighbours in $C$, and so
  $|N(B' \setminus A') \cap C|
  \le |B' \setminus A'|\cdot(\omega - 1)$.
  Moreover, $|B'| \le 2(\omega - 1)$ because
  $B'$ is contained in the neighbourhood of any vertex of $B$
  and $G$ is a quasi-line graph.
  Putting these inequalities together gives
  \begin{align*}
    \deg_{G^2}(b_0)
    \le (\omega - 2)|B'\setminus A'|+4\omega-2
    \le (\omega-2)(\omega-1)+4\omega-2 = \omega^2+\omega.
  \end{align*}
  So we have proved that every vertex in $B$ is degenerate.
  A similar argument proves that every vertex in $A$ is degenerate.
    We deduce the theorem for $v=a_1$ with $S = (A \cup B)\cap N(a_1)$.
    Indeed, the vertices of $N(v) \setminus S = A'$ are all adjacent to $a_2$.

\begin{itemize}
\item[\ref{second}] $G$ is a circular interval graph.
\end{itemize}
  By Lemma~\ref{lem:circular}, the square degree of any $v \in V$ satisfies
  $\deg_{G^2}(v) \le 4\omega-4$, which is at most $\omega^2+\omega$ since $\omega\ge3$.
 So it suffices to take $S=N(v)$ as every vertex is degenerate.
  
\begin{itemize}
\item[\ref{third}] $G$ is a composition of $k$ linear interval strips $(G_1', a_1', b_1'),\dots,(G_k', a_k', b_k')$ over $G_0$ a disjoint union of cliques on vertex set $\{a_1,\dots,a_k,b_1,\dots,b_k\}$.

\end{itemize}
Suppose $G_0$ is the disjoint union of $\ell$ cliques $C_1,\dots,C_\ell$.
  We have that $V = \bigcup_{i=1}^kV(G_i') \setminus \{a_i',b_i'\}$
  and that $G$ is the union of $G_i'\setminus\{a_i',b_i'\}$, $i \in \{1,\dots, k\}$,
  ``glued'' to the cliques on $C_1',\dots,C_\ell'$ defined by
  \begin{align*}
  C_j' := \left(\bigcup_{a_i' \in C_j} N_{G_i'}(a_i') \setminus \{b_i'\}\right)
  \bigcup \left(\bigcup_{b_i' \in C_j} N_{G_i'}(b_i') \setminus \{a_i'\}\right).
  \end{align*}

  Fix $i \in \{1,\dots,k\}$ and denote $C_{j_1}$ and $C_{j_2}$ the cliques
  (of $G_0$) such that $a_i \in C_{j_1}$ and $b_i \in C_{j_2}$.
  Let $H_i$ be the subgraph of $G$ induced by
  $C_{j_1}'\cup C_{j_2}'\cup (V(G_i')\setminus\{a,b\})$.

  We first observe the following three claims.
  \begin{claim}\label{claim:newstrip}
    There are $a_i'' \in C_{j_1}'$ and $b_i'' \in C_{j_2}'$
    such that $(a_i'',b_i'',H_i)$ is a linear interval strip.
  \end{claim}
  \begin{proof}
    Let $F_1,\dots,F_k$ be intervals homeomorphic to $[0,1]$
    and assume that the elements of $V(G_i')$ are real values such that the intervals
    $F_1,\dots,F_k$ represent $(a_i,b_i,G_i')$ as a linear interval strip.
    To construct a representation of $H_i$ as a linear interval strip,
    it suffices to keep the intervals $F_1,\dots,F_k$
    as well as the real values of the elements of $V(G_i')\setminus\{a_i,b_i\}$
    and assign the value of $a_i$ to all elements of
    $C_{j_1}' \setminus N_{G_i'}(a_i)$ and the value of $b_i$ to all elements
    of $C_{j_2}' \setminus N_{G_i'}(b_i)$.
    It then suffices to define $a_i''$ (resp.~$b_i''$) as one of the vertices
    with smallest (resp.~greatest) value.
  \end{proof}
  
\begin{claim}\label{claim:degenerate1}
If $v \in V(G_i') \setminus (\{a_i',b_i'\}\cup N_{G_i'}(a_i') \cup N_{G_i'}(b_i'))$ for some $i \in \{1,\dots, k\}$, then $v$ is degenerate.
\end{claim}
\begin{proof}
  Notice that $N_{G^2}(v) = N_{(H_i)^2}(v)$.
  Moreover, Claim~\ref{claim:newstrip} ensures that $H_i$ is a linear
  interval graph and thus a circular interval graph. So
  Lemma~\ref{lem:circular} yields
  \begin{align*}
    \deg_{G^2}(v) = \deg_{(H_i)^2}(v) \le 4\omega-4.
  \end{align*}
  This is at most $\omega^2+\omega$ since $\omega\ge3$.
\end{proof}
\begin{claim}\label{claim:degenerate2}
If $v$ is in $V(G_i') \setminus (\{a_i',b_i'\}\cup N_{G_i'}(a_i'))$ or in $V(G_i') \setminus (\{a_i',b_i'\}\cup N_{G_i'}(b_i'))$
 for some $i \in \{1,\dots, k\}$, then $v$ is degenerate.
\end{claim}
\begin{proof}
  By Claim~\ref{claim:degenerate1} and
  by symmetry of the roles played by $a_i'$ and $b_i'$, we may assume
  that $v$ is a neighbour of $a_i'$ but not of $b_i'$ in $G_i'$.
  Notice that in this case, $N_{G^2}(v)$ is contained in
  $N_{(H_i)^2}(v) \cup (N(C_{j_1}')\cap N^2(v)).$
  Claim~\ref{claim:newstrip} guarantees some $a_i'' \in C_{j_1}'$ and $b_i'' \in C_{j_2}'$ such that $(a_i'',b_i'',H_i)$ is a linear interval strip.
  Notice that $v \in C_{j_1}'$ is a neighbour of $a_i''$ in $G$ (and thus in $H_i$).
  Hence Lemma~\ref{lem:interval} applies and yields
  $\deg_{(H_i)^2}(v) \leq 3\omega-3$.
  
  By Lemma~\ref{lem:cliquesecond},
   $|N(C_{j_1}')\cap N^2(v)|
  \le (|C_{j_1}'|-1)(\omega-1) \le (\omega-1)^2$.
  So
  \begin{align*}
    \deg_{G^2}(v) &\le \deg_{(H_i)^2}(v) + |N(C_{j_1}')\cap N^2(v)|\\
    &\le 3\omega-3 + (\omega-1)^2 = \omega^2 +\omega -2
  \end{align*}
  which is less than $\omega^2 +\omega$.
\end{proof}

Case~\ref{third} now divides into three subcases.
   \begin{itemize}
     \setlength\itemindent{10pt}
   \item[\ref{third}(a)]
There exists $v \in V(G_i') \setminus (\{a_i',b_i'\}\cup N_{G_i'}(a_i') \cup N_{G_i'}(b_i'))$ for some $i \in \{1,\dots, k\}$.
   \end{itemize}
   
   Claims~\ref{claim:degenerate1} and~\ref{claim:degenerate2} imply that $v$ and every neighbour of $v$ is degenerate. So take $S=N(v)$.
  
  \begin{itemize}
    \setlength\itemindent{10pt}
  \item[\ref{third}(b)]
For every $i \in \{1,\dots, k\}$, $V(G_i') = N_{G_i'}(a_i') \cup N_{G_i'}(b_i') \cup \{a_i', b_i'\}$ and there exists $i_0$ and a vertex $v$ either in $V(G_{i_0}') \setminus (\{a_{i_0}',b_{i_0}'\}\cup N_{G_{i_0}'}(a_{i_0}'))$ or in $V(G_{i_0}') \setminus (\{a_{i_0}',b_{i_0}'\}\cup N_{G_{i_0}'}(b_{i_0}'))$.
  \end{itemize}
  
  Claims~\ref{claim:degenerate1} and~\ref{claim:degenerate2} imply that $v$ is degenerate and every neighbour of $v$ is either in the strip  $G'_{i_0}$
      (in which case it is degenerate) or in the clique $C_{j_1}'$. So take $S = N(v)\setminus C_{j_1}'$.

\begin{itemize}
\setlength\itemindent{10pt}
\item[\ref{third}(c)]
  For every $i \in \{1,\dots, k\}$,
  $V(G_i')\setminus \{a_i',b_i'\} = N_{G_i'}(a_i') \cap N_{G_i'}(b_i')$.
\end{itemize}

  In this case, the graph $G$ is the (non-disjoint) union
  of the cliques $C_j'$, $j \in \{1,\dots, \ell\}$
  because any edge of $G_i'\setminus \{a_i',b_i'\}$ is contained in the clique
  $C_j'$ such that $a_i \in C_j$.
  Moreover, each vertex $v \in V$ belongs to exactly
  two such cliques $C_{j_v^1}'$ and $C_{j_v^2}'$.
  It follows that $G$ is a line graph of a multigraph.
  In particular, $G$ is the line graph of the multigraph on vertex set $C_1,\dots,C_\ell$ with 
  an edge between $C_{j_v^1}'$ and $C_{j_v^2}'$ for each vertex $v \in V$.
\end{proof}


\section{Line graphs of multigraphs}\label{sec:linemulti}

In this section, we prove Theorem~\ref{thm:linemulti}.
Without loss of generality, we may assume hereafter that multigraphs are loopless.
Since we are now close enough to the original problem of Erd\H{o}s and Ne\v{s}et\v{r}il, let us recast Theorem~\ref{thm:linemulti} in terms of edge-colouring.
A \emph{strong edge-colouring} of a (multi)graph $F$ is a proper edge-colouring of $G$ such that any two edges with an edge between them are also required to have distinct colours.
The \emph{strong chromatic index} $\chi_s'(F)$ of $F$ is the smallest integer $k$ such that $F$ admits a strong edge-colouring using $k$ colours.

\begin{theorem}\label{thm:strongmulti}
There are some absolute constants $\eps>0$ and $\Delta_0$ such that $\chi_s'(F)\le (2-\eps)\Delta(F)^2$ for any multigraph $F$ with $\Delta(F) \ge \Delta_0$.
\end{theorem}
\noindent
Since $\chi_s'(F) = \chi(L(F)^2)$ and $\Delta(F)\le\omega(L(F))$ for any multigraph $F$, this implies Theorem~\ref{thm:linemulti}.
Indeed, due to the ``trivial'' upper bound $\chi_s'(F) \le 2\Delta(F)^2-2\Delta(F)+1$, it suffices to choose $\min\{\eps,3/\Delta_0\}$ for the constant certifying Theorem~\ref{thm:linemulti}.

It seems to us that allowing edges of large multiplicity tends to lead to a smaller strong chromatic index. For instance, given a multigraph $F=(V,E)$ with $\Delta(F) \le \Delta$, if $e$ is in an edge of multiplicity $\eps\Delta$, then easily we have that $\deg_{L(F)^2}(e) \le 2(1-\eps)\Delta^2+O(\Delta)$. So we do not need to consider any multigraph with an edge of multiplicity $3\Delta/8$ or more. It does, however, seem difficult in general to eliminate consideration of all those edges, say, of multiplicity two.

Rather, to prove Theorem~\ref{thm:strongmulti}, we take the tack of Molloy and Reed used to affirm the original question of Erd\H{o}s and Ne\v{s}et\v{r}il. We employ a bound on the chromatic number of graphs whose neighbourhoods are not too dense.
The following can be shown with the probabilistic method.

\begin{lemma}[Molloy and Reed~\cite{MoRe97}]\label{lem:MoRe}
For any $\eps > 0$, there exist $\delta > 0$ and $\Delta_0$ such that the following holds. For all $\Delta\ge\Delta_0$, if $G$ is a graph with $\Delta(G) \le \Delta$ and with at most $(1-\eps)\binom{\Delta}2$ edges in each neighbourhood, then $\chi(G) \le (1-\delta)\Delta$.
\end{lemma}

Since $\Delta(L(F)^2) \le 2\Delta(F)(\Delta(F)-1)$, we obtain Theorem~\ref{thm:strongmulti} 
by an application of Lemma~\ref{lem:MoRe} to $L(F)^2$, the validity of which is certified as follows.

\begin{lemma}\label{lem:multisparse}
There are absolute constants $\eps>0$ and $\Delta_0$ such that the following holds. For all $\Delta\ge\Delta_0$, if $F=(V,E)$ is a multigraph with $\Delta(F) \le \Delta$, then $N_{L(F)^2}(e)$ induces a subgraph of $L(F)^2$ with at most $(1-\eps)\binom{2\Delta(\Delta-1)}{2}$ edges for any $e\in E$. 
\end{lemma}
Molloy and Reed proved this for $F$ simple.
We have adapted their proof to account for edges of multiplicity.
The adaptation is mainly technical but not completely straightforward, so we include the proof details for completeness.
We remark that it is also possible to adapt a proof of Lemma~\ref{lem:multisparse} for $F$ simple recently given by Bruhn and Joos~\cite{BrJo15+} that yields an asymptotically extremal answer.
On the other hand, it is known that this approach, via Lemma~\ref{lem:MoRe}, is insufficient alone to yield the optimal constant $\eps=3/4$ in Theorem~\ref{thm:strongmulti} for $F$ a simple graph. So we have made no effort to look for a value better than what we obtained here.

\begin{proof}[Proof of Lemma~\ref{lem:multisparse}]
We specify the constant $\eps > 0$ as well as some other constants $\eps_1,\eps_2,\eps_3>0$ later in the proof. 
Let $F=(V,E)$ be a multigraph with $\Delta(F) \le \Delta$. 
Without loss of generality, we may assume that $F$ is $\Delta$-regular.

Let $e=u_1u_2\in E$. Let $A = N_F(u_1) \setminus \{u_2\}$, $B=N_F(u_2)
\setminus \{u_1\}$ and $C = N_F(A)\cup N_F(B)\setminus(A\cup B \cup \{u_1, u_2\})$.
Let $M$ be the set of edges between $u_1$ and $u_2$ in parallel with $e$. 
For a positive integer $i$, let $\Lambda_i$ be the collection of vertices $a\in A\cup B$ such that $|E_F(\{a\},\{u_1,u_2\})|=i$.
We treat three cases:
\begin{enumerate}
\item\label{case1}
$|E_F(A\cup B)|
+(2\Delta-1)|M|
+\sum_{i=2}^{\Delta} (i-1)\Delta  |\Lambda_i| 
 > \eps_1\Delta^2$,
\item\label{case2}
$\sum_{c\in C} |E_F(\{c\},A\cup B)| \cdot (\Delta- |E_F(\{c\},A\cup B)|)>\eps_2\Delta^3$, and
\item\label{case3}
we are neither in Case~\ref{case1} nor Case~\ref{case2}.
\end{enumerate}

\paragraph{Case~\ref{case1}.}

An exercise in double-counting checks that $\deg_{L(F)^2}(e)$ equals
\begin{align}\label{eqn:squaredegree}
2\Delta(\Delta-1) - 
\left(|E_F(A\cup B)|
+(2\Delta-1)|M|
+\sum_{i=2}^{\Delta} (i-1)\Delta  |\Lambda_i|\right)
\end{align}
and so $\deg_{L(F)^2}(e) < (2-\eps_1)\Delta^2$ in this case. Thus, $N_{L(F)^2}(e)$ necessarily induces a subgraph of $L(F)^2$ with fewer than $\deg_{L(F)^2}(e)^2/2$ edges, which implies
\begin{align}\label{eqn:case1}
|E_{L(F)^2}(N_{L(F)^2}(e))| < (2-2\eps_1+{\eps_1}^2/2)\Delta^4.
\end{align}

\paragraph{Case~\ref{case2}.}

For any $e_1 \in N_{L(F)^2}(e)$, note that $|N_{L(F)^2}(e_1)\cap N_{L(F)^2}(e)|$ is at most $2\Delta^2$ minus the number of three-edge walks in $F$ with first edge $e_1$ and last edge not in $N_{L(F)^2}(e)$. Every two-edge path $ae_2ce_3x$ in $F$, where $e_2,e_3\in E$, $a,c,x\in V$, $a\in A\cup B$, $c\in C$, and $x\notin A\cup B$, contributes $\Delta$ such three-edge walks in $F$. So the total number of such three-edge walks exceeds $\eps_2\Delta^4$ from which we conclude by the handshaking lemma that
\begin{align}\label{eqn:case2}
|E_{L(F)^2}(N_{L(F)^2}(e))| < (2-\eps_2/2)\Delta^4.
\end{align}

\paragraph{Case~\ref{case3}.}

We shall bound from above the number of edges of $L(F)^2$ induced by $N_{L(F)^2}(e)$ via a lower bound on the number of closed four-edge walks in $F$ that only use edges between $A\cup B$ and $C$.
For any $e_1 \in E_F(A\cup B,C)$, note that $|N_{L(F)^2}(e_1)\cap N_{L(F)^2}(e)|$ is at most $2\Delta^2$ minus the number of such closed four-edge walks to which it belongs.
It follows that the number of edges of $L(F)^2$ induced by $N_{L(F)^2}(e)$ is at most $2\Delta^4$ minus twice the number of such closed four-edge walks.
For $c_1,c_2\in C$, let $w(c_1,c_2)$ denote the number of two-edge walks between $c_1$ and $c_2$ with middle vertex in $A\cup B$. To be unambiguous about what this means when $c_1=c_2$, let us in this case only count (unordered) pairs of distinct edges that both have as endpoints both $c_1$ and some vertex in $A\cup B$. Note that the number of such closed four-edge walks is at least
\begin{align*}
\sum_{\{c_1,c_2\}\in \binom{C}{2}+C}\binom{w(c_1,c_2)}{2},
\end{align*}
where we have used the unconventional notation $\binom{X}{2}+X$ to denote the collection of all unordered pairs of distinct elements from $X$ together with all pairs $\{x,x\}$ for $x\in X$.

Let $C' = \{c\in C \,|\, |E_F(\{c\},A\cup B)| \ge \eps_3 \Delta\}$.
Using the expression in~\eqref{eqn:squaredegree}, it follows from the fact that we are not in Case~\ref{case1} that
\begin{align*}
|E_F(A\cup B,C)| 
& \ge \deg_{L(F)^2}(e) - |E_F(A\cup B)| - 2\Delta \\
& \ge (2-2\eps_1)\Delta^2-O(\Delta).
\end{align*}
Moreover, we have that
\begin{align*}
|E_F&(A\cup B, C\setminus C')| 
 = ((1-\eps_3)\Delta)^{-1} \sum_{c\in C\setminus C'} |E_F(\{c\},A\cup B)|\cdot (\Delta-\eps_3\Delta) \\
& \le ((1-\eps_3)\Delta)^{-1} \sum_{c\in C\setminus C'} |E_F(\{c\},A\cup B)|\cdot(\Delta-|E_F(\{c\},A\cup B)|) \\
& \le ((1-\eps_3)\Delta)^{-1} \eps_2 \Delta^3 = \frac{\eps_2}{1-\eps_3}\Delta^2,
\end{align*}
where the last inequality holds because we are not in Case~\ref{case2}. 
Thus
\begin{align}
|E_F(A\cup B, C')| 
& \ge \left(2-2\eps_1-\frac{\eps_2}{1-\eps_3}\right)\Delta^2-O(\Delta).
\end{align}

By applying Jensen's Inequality with respect to the convex function $\binom{x}{2}$,
\begin{align*}
\sum_{\{c_1,c_2\}\in \binom{C'}{2}+C'} &w(c_1,c_2)
 = \sum_{a\in A\cup B} \binom{|E_F(\{a\},C')|}{2}\\
 &\ge |A\cup B| \binom{|A\cup B|^{-1}\sum_{a\in A\cup B}|E_F(\{a\},C')|}{2}\\
 &\ge \frac{|E_F(A\cup B, C')|^2}{2|A\cup B|} -O(\Delta^2) \\
 &\ge \left(1-\eps_1-\frac{\eps_2}{2(1-\eps_3)}\right)^2\Delta^3 -O(\Delta^2),
\end{align*}
where we used $|A\cup B| \le 2\Delta$ in the last inequality.
We have that $|C'| \le 2\Delta/\eps_3$, so again by Jensen's Inequality with respect to $\binom{x}{2}$
\begin{align*}
\sum_{\{c_1,c_2\}\in \binom{C}{2}+C}&\binom{w(c_1,c_2)}{2} 
 \ge \sum_{\{c_1,c_2\}\in \binom{C'}{2}+C'}\binom{w(c_1,c_2)}{2}\\
& \ge \left(\binom{|C'|}{2}+|C'|\right)\binom{\left(\binom{|C'|}{2}+|C'|\right)^{-1}\sum_{\{c_1,c_2\}\in \binom{C'}{2}+C'}w(c_1,c_2)}{2} \\
& \ge \frac{{\eps_3}^2}{4} \left(1-\eps_1-\frac{\eps_2}{2(1-\eps_3)}\right)^4 \Delta^4 -O(\Delta^3).
\end{align*}
We conclude in this case that
\begin{align}\label{eqn:case3}
|E_{L(F)^2}(N_{L(F)^2}(e))| \le \left(2-\frac{{\eps_3}^2}{2} \left(1-\eps_1-\frac{\eps_2}{2(1-\eps_3)}\right)^4\right)\Delta^4 +O(\Delta^3).
\end{align}

\medskip
Considering~\eqref{eqn:case1}--\eqref{eqn:case3}, we obtain a bound on $|E_{L(F)^2}(N_{L(F)^2}(e))|$ that is a nontrivial factor smaller than $\binom{2\Delta(\Delta-1)}{2}\sim 2\Delta^4$ in all three cases, provided $\Delta$ is large enough and provided that we can find $\eps_1,\eps_2,\eps_3>0$ such that
\begin{align*}
-2\eps_1+\frac{{\eps_1}^2}{2} < 0 \text{ and }
1-\eps_1-\frac{\eps_2}{2(1-\eps_3)} > 0.
\end{align*}
Exactly the same choices as made by Molloy and Reed, $\eps_1 = 1/30$, $\eps_2 = 1/9$, $\eps_3 = 2/3$, suffice here for $\eps = 1/36$, thus completing the proof.
\end{proof}


\section{Proof of Theorem~\ref{thm:main}}\label{sec:main}

Let $\eps$ be the constant given by Theorem~\ref{thm:linemulti}. Possibly by decreasing this choice of $\eps$, we may assume due to Proposition~\ref{prop:trivial} that the result holds for $\omega(G) < 6$. So it only remains to consider $\omega(G)\ge6$. Since $\eps$ cannot be greater than $3/4$, we know that $\omega(G)^2+\omega(G)+1$ and $\omega(G)^2+(\omega(G)+1)/2$ are at most $\lfloor(2-\eps)\omega(G)^2\rfloor$.

We proceed by induction on the number of vertices in $G$. The base case of $G$ having $6$ vertices is trivially true. So now assume $G$ has more than $6$ vertices and that the result holds for all graphs with fewer vertices than $G$ has. We have three cases to consider in succession.

If $G$ is the line graph of a multigraph, then the result follows from Theorem~\ref{thm:linemulti}.

If $G$ is not the line graph of a multigraph but is a quasi-line graph, then let $v$ be the vertex and $S$ the set given by Theorem~\ref{thm:quasiline}. Since $G\setminus v$ is a claw-free graph and $\omega(G\setminus v) \le \omega(G)$, it follows by induction that there is a proper colouring of $(G\setminus v)^2$ with colours from ${\mathcal K} = \{1,\dots,\lfloor(2-\eps)\omega(G)^2\rfloor\}$. We shall use this colouring to obtain a proper colouring of $G^2$.  To do so, we first uncolour the vertices of $S$ and then recolour them with distinct colours from ${\mathcal K}$ as follows. For each $u\in S$, we need to provide a colour in $\mathcal K$ that is distinct not only from the colours appearing on the vertices of $N_{G^2}(u)\setminus (\{v\}\cup N_G(v)) = N_{(G\setminus v)^2}(u)\setminus N_G(v)$, but also from the $\deg_G(v)-|S|$ colours assigned to $N_G(v)\setminus S$. Since $\deg_{G^2}(u) \le \omega(G)^2+\omega(G)$ and $\{v\}\cup N_G(v)\setminus\{u\} \subseteq N_{G^2}(u)$, the number of colours from $\mathcal K$ potentially available to $u$ in this sense (just after recolouring) is at least $|{\mathcal K}| - (\omega(G)^2+\omega(G))+|S|\ge |S|$. By this fact, it follows that we can greedily recolour $S$ with distinct colours from $\mathcal K$ to obtain a proper colouring of $(G\setminus v)^2$ in which all the vertices in $N_G(v)$ have different colours. Now, since $\deg_{G^2}(v) \le \omega(G)^2+\omega(G)$, there is at least one colour available in $\mathcal K$ different from all the colours in $N_{G^2}(v)$. Giving such a colour to $v$ yields a proper colouring of $G^2$ from $\mathcal K$.

If $G$ is not a quasi-line graph, then let $v$ be the vertex given by Theorem~\ref{thm:clawfree}. Again by induction there is a proper colouring of $(G\setminus v)^2$ with $\lfloor(2-\eps)\omega(G)^2\rfloor$ colours. In this colouring, necessarily all the vertices in $N_G(v)$ have different colours. Since $\deg_{G^2}(v) \le \omega(G)^2+(\omega(G)+1)/2$, there is at least one colour available to $v$, which leads to the desired proper colouring of $G^2$.
\qed


\section{Multigraphs with induced matching number one}\label{sec:diametertwo}

In this section, we outline how to establish Theorem~\ref{thm:diametertwo}. For any graph $G=(V,E)$, if $\omega(G^2)=|V|$, then the minimum square degree of $G$ must be at least $|V|-1$. So by Theorems~\ref{thm:clawfree} and~\ref{thm:quasiline}, it suffices to show Theorem~\ref{thm:diametertwo} in the special case of $G$ being the line graph of a multigraph. This is implied by the following theorem (which is slightly stronger than what we require).
Chung, Gy\'arf\'as, Tuza and Trotter~\cite{CGTT90} proved this for $F$ a simple graph.
We adopt their notation. If $\Delta\ge 2$ is even, then $C_5(\Delta)$ is the graph obtained from the five-cycle by substituting each vertex by a stable set of size $\Delta/2$. If $\Delta\ge 3$ is odd, then $C_5(\Delta)$ is the graph obtained by substituting two consecutive vertices on the five-cycle by stable sets of size $(\Delta+1)/2$ and the remaining three by stable sets of size $(\Delta-1)/2$. Let $f(\Delta)$ denote the number of edges in $C_5(\Delta)$, so it is $5\Delta^2/4$ if $\Delta$ is even and $(5\Delta^2-2\Delta+1)/4$ if $\Delta$ is odd.
\begin{theorem} \label{thm:multiCGTT}
  Let $\Delta \ge 2$ and suppose $F=(V,E)$ is a multigraph with maximum degree
  $\Delta(F)\le\Delta$ and whose underlying simple graph $F_0$ is connected
  and induces no $2K_2$.
\begin{enumerate}
\item\label{it:CGTT:1} If $F$ is bipartite, then $|E| \le \Delta^2$. 
  Equality holds if and only if $F$ is the complete bipartite graph $K_{\Delta,\Delta}$.
\item\label{it:CGTT:2} If $\omega(F)=2$ and $F$ is not bipartite, then $|E| \le f(\Delta)$.
  Equality holds if and only if $F$ is isomorphic to $C_5(\Delta)$. 
\item\label{it:CGTT:3} If $\omega(F)\ge5$, then $|E| < f(\Delta)$. 
\item\label{it:CGTT:4} If $\omega(F)=4$, then $|E| < f(\Delta)$.
\item\label{it:CGTT:5} If $\omega(F)=3$, then $|E| < f(\Delta)$.
\end{enumerate}
\end{theorem}

The original proof in~\cite{CGTT90} for simple graphs extends to multigraphs with some minor modifications. For brevity, we have elected to include only an outline for most of these modifications, and refer liberally to~\cite{CGTT90}.

\begin{proof}[Proof of~\ref{it:CGTT:1}]
  Let $A$ and $B$ be the color classes of $F$.
  By the corollary of Theorem~1 in~\cite{CGTT90} applied to $F_0$,
  there is a vertex $v$ adjacent to all vertices of $A$,
  so $|A| \le \Delta$. It follows that $|E| \le |A| \cdot \Delta \le \Delta^2$.

  By the above argument, equality is possible
  only if $|A|=|B|=\Delta$ and $F$ is $\Delta$-regular.
  Let us prove by induction on $\Delta$
  that equality holds only when $F$ is the simple graph $K_{\Delta,\Delta}$.
  This is clear for $\Delta = 1$.
  Assume $\Delta \ge 2$ and $|E(F)| = \Delta^2$.
  The corollary of Theorem~1 in~\cite{CGTT90}
  applied to $F_0$ gives $a \in A$ and $b \in B$ with neighbourhoods
  equal to $B$ and $A$ respectively.
  Since $\deg(a) = \Delta = |B|$ and similarly $\deg(b) = |A|$,
  the vertices $a$ and $b$ have no incident multiedges.
  Consequently, the multigraph $F'=F \setminus \{a,b\}$
  has $(\Delta-1)^2$ edges and no induced $2K_2$ in its underlying simple graph, and
  is bipartite, $(\Delta-1)$-regular and
  connected. The induction hypothesis applied to $F'$
  gives $F'=K_{\Delta-1,\Delta-1}$, and so $F=K_{\Delta,\Delta}$.
\end{proof}

\begin{proof}[Proof of~\ref{it:CGTT:2}]
  By Theorem~2 in~\cite{CGTT90} applied to $F_0$,
  we know that $F_0$ is the blow-up of a $C_5$ by stable sets
  $A_1,\dots,A_5$ of respective sizes $a_1,\dots,a_5$.
  For each $i$, the maximum degree condition on $v \in A_i$ gives
  \begin{equation}\label{eq:caseii}
    a_{i-1} + a_{i+1} = \deg_{F_0}(v) \le \deg_F(v) \le \Delta
  \end{equation}
  where the subscript are taken modulo $5$.
  Summing~\eqref{eq:caseii} over $i\in\{1,\dots,5\}$
  gives
  \begin{equation}\label{eq:caseii:sum}
    2\sum_{i=1}^5a_i \le 5 \Delta.
  \end{equation}
  We distinguish two cases depending on the parity of $\Delta$.  

  If $\Delta$ is even, then $|E| \le \frac{\Delta}{2} \cdot \sum_{i=1}^5a_i \le \frac{5}{4}\Delta^2 = f(\Delta)$.
  Inequality~\eqref{eq:caseii} is an equality
  if and only if every $v \in A_i$ is incident to exactly $\Delta$
  simple edges in $F$ and $\Delta=a_{i-1} + a_{i+1}$.
  Moreover, the only solution of this system is $a_i = \frac{\Delta}{2}$.
  To see this, compute
  $$2a_i = \sum_{j=0}^4(-1)^j (a_{i + 2j} + a_{i + 2j + 2}) =
  \sum_{j=0}^4(-1)^j\Delta = \Delta.$$

  If $\Delta$ is odd, then~\eqref{eq:caseii:sum} improves to
  $2\sum_{i=1}^5a_i \le 5 \Delta - 1$.
  We may even assume that $2\sum_{i=1}^5a_i = 5 \Delta - 1$.
  Indeed, if $2\sum_{i=1}^5a_i \le 5 \Delta - 2$,
  then $|E| \le \frac{\Delta}{4}(5 \Delta - 2) < f(\Delta)$.
  It follows that~\eqref{eq:caseii}
  is an equality for all indices $i$ in $\{1,\dots,5\}$ except one.
  Without loss of generality, we may assume $a_{i - 1} + a_{i + 1} = \Delta$
  for $i \in \{1,2,3,4\}$ and $a_4 + a_1 = \Delta - 1$.
  The only solution to this system is given
  by $a_2=a_3=a_5=(\Delta + 1)/2$ and $a_1 = a_4 = (\Delta - 1)/2$,
  which happens if and only if $F_0 = C_5(\Delta)$.
  Noticing that doubling any edge of $C_5(\Delta)$
  creates a vertex of degree $\Delta + 1$,
  we conclude that $F=F_0=C_5(\Delta)$.
\end{proof}

Modulo the fact that we use $F$ instead of $G$ and $\Delta$ instead of $D$,
we use the same notation and definitions as in~\cite{CGTT90}
with the exception that for $y_1, y_2 \in Y$,
the weight $w(y_1,y_2)$ denotes the number
of \emph{edges} from $y_1$ to $K$
plus the number of \emph{edges} from $y_2$ to $K$ in $F$
(instead of the number of neighbours).
Claim~0 in~\cite{CGTT90} is satisfied by this new definition.

\begin{proof}[Sketch proof of~\ref{it:CGTT:3}]
  As there are at most $p\Delta - 2|E(K)|$ edges from $K$ to $Y$,
  relation~$(*)$ in~\cite{CGTT90} becomes
  \begin{equation}\label{eq:newstar}
    \sum_{e \in Y} w(e) \le (p\Delta - 2|E(K)|)(\Delta-1)
  \end{equation}
  (where multiple edges appear multiple times in the sum).
  Using Claim~0 in~\cite{CGTT90},
  \begin{align}
    |E| & \le p\Delta - |E(K)| + |E(Y)|\nonumber\\
    & \le p\Delta - |E(K)| + \frac{1}{p-1}(\Delta-1)(p\Delta - 2|E(K)|) \label{eq:iii}\\
    & = p\Delta + \frac{p}{p-1}(\Delta-1)\Delta - |E(K)|\left(1 + 2\frac{\Delta-1}{p-1}\right).\nonumber
  \end{align}
  Using that $|E(K)| \ge \binom{p}{2}$,
  \begin{align*}
    |E| < p\Delta + \frac{p}{p-1} \Delta^2 - \frac{p(p-1)}{2} - p(\Delta-1)
    \le \frac{p}{p-1} \Delta^2 - \frac{p(p-3)}{2}.
  \end{align*}
  This finishes the proof for $p \ge 5$.
\end{proof}

\begin{proof}[Sketch proof of~\ref{it:CGTT:4}]
  We proceed as in~\cite{CGTT90}. Using the same trick
  with the set $E_3 = \{e \in E \,|\, w(e) = 3 \}$, \eqref{eq:iii}
  can be improved to
  \begin{align*}
    |E| & \le 4\Delta - |E(K)| + \frac{1}{4}(\Delta-1)(4\Delta - 2|E(K)|) + \frac{1}{4}|E_3|\\
    &= \Delta^2-3\Delta - |E(K)|\left(1 + \frac{\Delta-1}{2}\right) + \frac{1}{4}|E_3|\\
    &\le \Delta^2-3\Delta - 6\left(1 + \frac{\Delta-1}{2}\right)  + \frac{1}{4}|E_3|\\
    &\le \Delta^2 - 3 + \frac{1}{4}|E_3|.
  \end{align*}
  By the same structural arguments, each $e \in E_3$ has an endpoint in $A^1$
  (but the other one can also be in $A^1$) and
  it follows that $|E_3| \le (\Delta-1)(\Delta-2)$.  
\end{proof}

\begin{proof}[Sketch proof of~\ref{it:CGTT:5}]
  The original proof applies just as well to line graphs of multigraphs.
  We only need to check the following bounds:
  \begin{enumerate}\renewcommand{\theenumi}{(\alph{enumi})}
  \item\label{it:1} $|E(Y)| \le |Y|(\Delta-1)/2$;
  \item\label{it:2} the number of edges from $K$ to $Y$
    is at most $3\Delta - 6$; and
  \item\label{it:3} $|E| \le |Y|(\Delta-1)/2 + 3\Delta - 3$.
  \end{enumerate}
  The inequality in~\ref{it:1} follows from the fact that each $v \in Y$ has at least one
  neighbour in $K$. The number of edges from
  $K$ to $Y$ is at most $3\Delta - 2|E(K)|$.
  This first implies~\ref{it:2} as $|E(K)| \ge 3$.
  Second, together with~\ref{it:1} it gives
  $|E| \le |Y|(\Delta-1)/2 + 3\Delta - |E(K)|$, which implies~\ref{it:3} using again that $|E(K)| \ge 3$.

  The multigraph analogues of Claims~1 to~7 in~\cite{CGTT90} can be proved using the above three properties
  as axioms in addition to some structural considerations
  that apply exactly in the same way to multigraphs.
  
  More precisely, we have the following in~\cite{CGTT90}: Claim~1 relies on~\ref{it:3};
  Claim~2 only relies on Claim~1;
  Claim~3 only uses Claims~1 and~2 and~\ref{it:2};
  Claim~4 uses~\ref{it:3};
  Claim~5 uses Claims~2 and~4;
  Claim~6 has a purely structural proof; and
  Claim~7 uses Claims~2,~5 and~6.
  The conclusion only uses these claims.
\end{proof}

\subsection*{Acknowledgement}

We thank Luke Postle for alerting us to a subtlety in our original derivation of Theorem~\ref{thm:main}.

This research was begun during a visit of the first and third authors to Radboud University Nijmegen in June 2016.

This research was supported by a Van Gogh grant, reference 35513NM and by ANR project STINT, reference ANR-13-BS02-0007.

The second author is currently supported by a NWO Vidi Grant, reference 639.032.614.


\bibliographystyle{abbrv}
\bibliography{squarecf}

\begin{thebibliography}{10}

\bibitem{AKS80}
M.~Ajtai, J.~Koml{\'o}s, and E.~Szemer{\'e}di.
\newblock A note on {R}amsey numbers.
\newblock {\em J. Combin. Theory Ser. A}, 29(3):354--360, 1980.

\bibitem{BrJo15+}
H.~{Bruhn} and F.~{Joos}.
\newblock {A stronger bound for the strong chromatic index}.
\newblock {\em ArXiv e-prints}, Apr. 2015.

\bibitem{CaKa16+}
W.~{Cames van Batenburg} and R.~J. {Kang}.
\newblock {Squared chromatic number without claws or large cliques}.
\newblock {\em ArXiv e-prints}, 2016.

\bibitem{ChOv07}
M.~Chudnovsky and A.~Ovetsky.
\newblock Coloring quasi-line graphs.
\newblock {\em J. Graph Theory}, 54(1):41--50, 2007.

\bibitem{ChSe05}
M.~Chudnovsky and P.~Seymour.
\newblock The structure of claw-free graphs.
\newblock In {\em Surveys in combinatorics 2005}, volume 327 of {\em London
  Math. Soc. Lecture Note Ser.}, pages 153--171. Cambridge Univ. Press,
  Cambridge, 2005.

\bibitem{ChSe10}
M.~Chudnovsky and P.~Seymour.
\newblock Claw-free graphs {VI}. {C}olouring.
\newblock {\em J. Combin. Theory Ser. B}, 100(6):560--572, 2010.

\bibitem{ChSe12}
M.~Chudnovsky and P.~Seymour.
\newblock Claw-free graphs. {VII}. {Q}uasi-line graphs.
\newblock {\em J. Combin. Theory Ser. B}, 102(6):1267--1294, 2012.

\bibitem{CGTT90}
F.~R.~K. Chung, A.~Gy{\'a}rf{\'a}s, Z.~Tuza, and W.~T. Trotter.
\newblock The maximum number of edges in {$2K_2$}-free graphs of bounded
  degree.
\newblock {\em Discrete Math.}, 81(2):129--135, 1990.

\bibitem{Edm65}
J.~Edmonds.
\newblock Paths, trees, and flowers.
\newblock {\em Canad. J. Math.}, 17:449--467, 1965.

\bibitem{EOSV08}
F.~Eisenbrand, G.~Oriolo, G.~Stauffer, and P.~Ventura.
\newblock The stable set polytope of quasi-line graphs.
\newblock {\em Combinatorica}, 28(1):45--67, 2008.

\bibitem{Erd88}
P.~Erd{\H{o}}s.
\newblock Problems and results in combinatorial analysis and graph theory.
\newblock In {\em Proceedings of the {F}irst {J}apan {C}onference on {G}raph
  {T}heory and {A}pplications ({H}akone, 1986)}, volume~72, pages 81--92, 1988.

\bibitem{ErSz35}
P.~Erd{\H{o}}s and G.~Szekeres.
\newblock A combinatorial problem in geometry.
\newblock {\em Compos. Math.}, 2:463--470, 1935.

\bibitem{FOGS14}
Y.~Faenza, G.~Oriolo, and G.~Stauffer.
\newblock Solving the weighted stable set problem in claw-free graphs via
  decomposition.
\newblock {\em J. ACM}, 61(4):Art. 20, 41, 2014.

\bibitem{Gup66}
R.~Gupta.
\newblock The chromatic index and the degree of a graph.
\newblock {\em Notices Amer. Math. Soc.}, 13:719, 1966.

\bibitem{Kie89}
H.~A. Kierstead.
\newblock Applications of edge coloring of multigraphs to vertex coloring of
  graphs.
\newblock {\em Discrete Math.}, 74(1-2):117--124, 1989.
\newblock Graph colouring and variations.

\bibitem{Kim95}
J.~H. Kim.
\newblock The {R}amsey number {$R(3,t)$} has order of magnitude {$t^2/\log t$}.
\newblock {\em Random Structures Algorithms}, 7(3):173--207, 1995.

\bibitem{KiRe13}
A.~D. King and B.~Reed.
\newblock Asymptotics of the chromatic number for quasi-line graphs.
\newblock {\em J. Graph Theory}, 73(3):327--341, 2013.

\bibitem{KiRe15}
A.~D. King and B.~A. Reed.
\newblock Claw-free graphs, skeletal graphs, and a stronger conjecture on
  {$\omega$}, {$\Delta$}, and {$\chi$}.
\newblock {\em J. Graph Theory}, 78(3):157--194, 2015.

\bibitem{Min80}
G.~J. Minty.
\newblock On maximal independent sets of vertices in claw-free graphs.
\newblock {\em J. Combin. Theory Ser. B}, 28(3):284--304, 1980.

\bibitem{MoRe97}
M.~Molloy and B.~Reed.
\newblock A bound on the strong chromatic index of a graph.
\newblock {\em J. Combin. Theory Ser. B}, 69(2):103--109, 1997.

\bibitem{NaTa01}
D.~Nakamura and A.~Tamura.
\newblock A revision of {M}inty's algorithm for finding a maximum weight stable
  set of a claw-free graph.
\newblock {\em J. Oper. Res. Soc. Japan}, 44(2):194--204, 2001.

\bibitem{Sbi80}
N.~Sbihi.
\newblock Algorithme de recherche d'un stable de cardinalit\'e maximum dans un
  graphe sans \'etoile.
\newblock {\em Discrete Math.}, 29(1):53--76, 1980.

\bibitem{Sum81}
D.~P. Sumner.
\newblock Subtrees of a graph and the chromatic number.
\newblock In {\em The theory and applications of graphs ({K}alamazoo, {M}ich.,
  1980)}, pages 557--576. Wiley, New York, 1981.

\bibitem{Viz64}
V.~G. Vizing.
\newblock On an estimate of the chromatic class of a {$p$}-graph (in
  {R}ussian).
\newblock {\em Diskret. Analiz}, 3:25--30, 1964.

\end{thebibliography}

\bigskip

\small
\bigskip
\noindent
{Universit\'e Grenoble Alpes, CNRS, Grenoble INP, Laboratoire G-SCOP, Grenoble, France.
Email: \url{{remi.deverclos,lucas.pastor}@g-scop.grenoble-inp.fr}.}

\bigskip
\noindent
{Radboud University Nijmegen, Netherlands.
Email: \url{ross.kang@gmail.com}.}
%

\end{document}